\documentclass[12pt]{article}

\usepackage{amsmath,amsthm,amssymb}

\usepackage{hyperref}

\theoremstyle{plain}
\newtheorem{theorem}{Theorem}
\newtheorem{lemma}[theorem]{Lemma}
\newtheorem{corollary}[theorem]{Corollary}
\newtheorem{proposition}[theorem]{Proposition}

\theoremstyle{definition}
\newtheorem{definition}[theorem]{Definition}
\newtheorem{example}[theorem]{Example}

\theoremstyle{remark}

\newcommand{\N}{\mathbb{N}}
\newcommand{\Z}{\mathbb{Z}}
\newcommand{\Fact}{\mathrm{Fact}^+}
\newcommand{\Pref}{\mathrm{Pref}^+}
\newcommand{\Ret}{\mathcal{R}}
\newcommand{\PP}{\mathcal{P}}
\newcommand{\PPm}{\mathcal{P}_{\mathrm{m}}}
\newcommand{\PPs}{\mathcal{P}_{\mathrm{s}}}

\title{Factor Colorings of Linearly Recurrent Words}

\author{Ville Salo\\
\small Center for Mathematical Modeling\\[-0.8ex]
\small University of Chile\\[-0.8ex]
\small Santiago, Chile\\
\small\tt vsalo@dim.uchile.cl
\and
Ilkka T\"orm\"a\\
\small TUCS -- Turku Centre for Computer Science\\[-0.8ex]
\small University of Turku\\[-0.8ex]
\small Turku, Finland\\
\small\tt iatorm@utu.fi}

\date{\small Mathematics Subject Classifications: 68R15, 05D10}

\begin{document}

\maketitle

\begin{abstract}
In this short article, we study factor colorings of aperiodic linearly recurrent infinite words.
We show that there always exists a coloring which does not admit a monochromatic factorization of the word into factors of increasing lengths.
\end{abstract}


\section{Introduction}

Consider an infinite word $x$ over a finite alphabet $A$, and let $c$ be a function from the finite factors of $x$ to some finite set of colors.
We can view it as a coloring of $2$-subsets of the natural numbers, where the color of $\{i, j\}$ for $i < j$ is the color of the factor $x_{[i, j-1]}$.
The infinite version of Ramsey's theorem states that there is an infinite $c$-monochromatic subset of $\N$.
If we enumerate this subset as $\{i_0, i_1, i_2, \ldots\}$, this means that $x$ has a factorization $x = x_{[0, i_0-1]} x_{[i_0, i_1-1]} x_{[i_1, i_2-1]} \cdots$ where every factor except the first has the same color.

Now, a natural question arises: can we get rid of the offending first factor?
That is, when can we guarantee that a complete monochromatic factorization exists?
This question was raised in \cite{Za10}, and the first general treatment was given in \cite{LuPrZa14}, where it was proved for several important classes of infinite words that monochromatic factorizations need not exist.
In fact, the authors conjectured that all aperiodic words have this property.

One well known class of words not considered in the study is that of fixed points of primitive substitutions.
We consider a strictly larger class of words, the \emph{linearly recurrent} ones, which have been studied in \cite{DuHoSk99,Du00}, among others.
In this short article, we show that every aperiodic linearly recurrent word admits a factor coloring with no monochromatic \emph{monotone} factorization, that is, one where the lengths of the factors are increasing.
Note that the almost monochromatic factorization given by Ramsey's theorem can be made monotone, since any two adjacent factors (except the first) can be merged together to produce a longer factor with the same color.

\section{Definitions and Preliminary Results}

In this article, the set of integers includes zero: $\N = \{0, 1, 2, \ldots\}$.
The indexing of finite and infinite words also starts at $0$.
We denote by $A^\omega$ the right-infinite words over an alphabet $A$, while $A^*$, $A^+$, $A^n$ and $A^{\leq n}$ denote words of any finite length, length at least $1$, length exactly $n$ and length at most $n$, respectively.

Let $x \in A^* \cup A^\omega$ be a finite or right-infinite word, and let $\Fact x, \Pref x \subset A^+$ be the sets of nonempty factors and nonempty prefixes of $x$, respectively.
A \emph{factor coloring of $x$} is a function $c : \Fact x \to C$, where $C$ is a finite set of \emph{colors}.
A \emph{factorization} of $x$ is a (finite if $x \in A^*$) sequence $(U_i)_i$ of words in $\Fact x$ such that $x = \prod_i U_i$, where the product denotes concatenation, and it is \emph{monotone} if $|U_i| \leq |U_j|$ for all $i \leq j$.
Another factorization $(V_j)_j$ is \emph{coarser} than $(U_i)_i$ if for all $j$ there exists $i$ such that $|V_0 \cdots V_j| = |U_0 \cdots U_i|$.
The factorization $(U_i)_i$ is \emph{$c$-monochromatic} if $c(U_i) = c(U_j)$ holds for all $i, j$, and \emph{strongly $c$-monochromatic} if all factorizations coarser than it are also $c$-monochromatic.

\begin{definition}
For $h, k \geq 1$, let $\PP(h,k)$ be the class of infinite words $x \in A^\omega$ such that there exists a factor coloring $c : \Fact x \to C$ with $|C| \leq k$ such that no prefix of $x$ has a $c$-monochromatic factorization of length $h$.
We also define the classes $\PPm(h,k)$, which only considers monotone factorizations, and $\PPs(h,k)$, which only considers strongly $c$-monochromatic factorizations.

We denote $\PP(k) = \PP(\infty,k)$ and $\PP = \bigcup_{k \geq 1} \PP(k)$, and analogously for the other classes.
\end{definition}

Observe that we trivially have $\PP(h,k) \subset \PPm(h,k) \subset \PPs(h,k)$ for all $h,k \geq 1$.
Furthermore, $\PP(h, k) \subset \PP(h', k')$ holds whenever $h \leq h'$ and $k \leq k'$, and similarly for the other two classes.
Of course, the degenerate cases $h = 1$ and $k = 1$ give rise to empty classes.

\begin{example}
Let $x \in A^\omega$ be any infinite word and $u \in A^+$ any nonempty finite word.
We show that $u^h x \notin \PPm(h,k)$ holds for all $h, k \geq 1$.
Namely, if $c : \Fact u^h x \to C$ is any factor coloring of $u^h x$, then $(u, u, \ldots, u)$ (with $h$ copies of $u$) is a monotone monochromatic factorization of length $h$ of a prefix of $x$.

An analogous but somewhat weaker result holds for the classes $\PPs(h,k)$.
Let $H \in \N$ be such that every $k$-colored complete graph of size larger than $H$ contains a monochromatic induced subgraph of size $h + 1$; its existence follows from the finite version of Ramsey's theorem.
We claim that $y = u^H x \notin \PPs(h,k)$, and for that, let $c : \Fact y \to C$ be a factor coloring with $|C| \leq k$.
Let $G$ be the complete graph whose vertices are the integers $V = \{0, |u|, \ldots, H |u|\}$, and let $c'$ be the edge coloring of $G$ given by $c'(\{i,j\}) = c(y_{[i, j - 1]})$ for all $i < j \in V$.
By the definition of $H$, there exists a monochromatic subset $W \subset V$ of size $h+1$; denote $W = \{n_0, n_1, \ldots, n_h\}$ with $n_i < n_{i+1}$.
Since $c'(\{i,j\}) = c'(\{i-|u|,j-|u|\})$ holds for all applicable $i$ and $j$, we may assume $n_0 = 0$.
But then $(y_{[n_i, n_{i+1}-1]})_{0 \leq i < h}$ is a size-$h$ strongly monochromatic factorization of a prefix of $y$.
\end{example}

We remark that we have not tried to obtain any bounds for the number $H$ in the above example, except those given by Ramsey's theorem.
Since the edge coloring $c'$ of $G$ is of a special form (all edges of the same length have the same color), a more specialized approach would probably yield better bounds.

The following result, which was already mentioned in the introduction, is the motivation behind the research direction started in \cite{LuPrZa14}.
It was not stated for strongly $c$-monochromatic partitions, but the proof readily gives this stronger result.

\begin{proposition}[Proposition 3.1 in \cite{LuPrZa14}]
Let $c : \Fact x \to C$ be a factor coloring of an infinite word $x \in A^\omega$.
Then there exists $i \in \N$ such that the infinite tail $x_{[i, \infty)} \in A^\omega$ has a strongly $c$-monochromatic factorization.
\end{proposition}

\begin{proof}
Let $c'$ be a function from the $2$-subsets of $\N$ to the set $C$ defined by $c'(\{i, j\}) = c(x_{[i, j-1]})$ for all $i < j$.
By the infinite version of Ramsey's theorem, there exists an infinite $c'$-monochromatic set $S \subset \N$.
Denoting $S = \{i_0, i_1, i_2, \ldots\}$ with $i_j < i_{j+1}$, this implies that every factor $x_{[i_j, i_{j+k}-1]}$ has the same $c$-image for all $j \in \N$ and $k \geq 1$.
But then the infinite tail $x_{[i_0, \infty)}$ has the factorization $(x_{[i_j, i_{j+1}-1]})_{j \in \N}$, which is strongly $c$-monochromatic, and we choose $i = i_0$.
\end{proof}

In \cite{LuPrZa14}, it was proved (among other results) that all words that are not uniformly recurrent are in $\PP$, all words with full factor complexity ($\Fact x = A^+$) are in $\PP(2)$, and all Sturmian words are in $\PP(3)$.
Recall that Sturmian words are exactly the infinite words over the alphabet $\{0, 1\}$ that are \emph{balanced} (the number of $1$s in any two factors of the same length differs by at most $1$).
The authors conjectured that all aperiodic infinite words are in the class $\PP$, and by the results stated above, all still unsolved cases are uniformly recurrent.

The parameter $h$ and the classes $\PPm$ and $\PPs$ were not considered in \cite{LuPrZa14}, and some of their basic properties are still unknown.
For example, it is clear that $\bigcup_{h \geq 1} \PP(h,k) \subset \PP(\infty, k)$ holds for all $k \geq 1$, and analogously for the other classes, but we do not know whether the inclusion is strict (except in the degenerate case $k = 1$).

Let $x \in A^\omega$ be an infinite word.
For $u \in \Fact x$, a \emph{return to $u$ (in $x$)} is a word $w \in A^+$ such that $w u \in \Fact x$ and $u$ occurs exactly twice in $w u$.
The set of returns to $u$ in $x$ is denoted $\Ret_x(u)$.
We say that $x$ is \emph{uniformly recurrent}, if for every factor $u \in \Fact x$, the set $\Ret_x(u)$ is nonempty and finite.
Equivalently, every factor of $x$ occurs infinitely many times and with bounded gaps.
If $x$ is uniformly recurrent, then for each nonempty prefix $u$ of $x$, we have a unique factorization $x = r_0 r_1 r_2 \cdots$ where $r_i \in \Ret_x(u)$ for all $i \in \N$.
We fix some enumeration of $\Ret_x(u)$, which defines a bijection $\Theta_x^u : R_x(u) \to \Ret_x(u)$, where $R_x(u) = \{0, 1, \ldots, |\Ret_x(u)|-1\}$.
We extend $\Theta_x^u$ into a morphism from $R_x(u)^*$ to $A^*$ by $\Theta_x^u(n_0 n_1 \cdots n_{k-1}) = \Theta_x^u(n_0) \Theta_x^u(n_1) \cdots \Theta_x^u(n_{k-1})$.
The following lemma can be found in \cite{Du98}.

\begin{lemma}[Proposition 2.6 in \cite{Du98}]
\label{lem:Return}
Let $x \in A^\omega$ be a uniformly recurrent word, and let $u, v \in \Pref x$ with $|v| \leq |u|$.
\begin{enumerate}
\item The morphism $\Theta_x^u : R_x(u)^* \to A^*$ is injective.
\item We have $\Ret_x(u) \subset \Theta_x^v(R_x(v)^*)$.
\item There is a unique morphism $\lambda_v^u : R_x(u)^* \to R_x(v)^*$ such that $\Theta_x^v \circ \lambda_v^u = \Theta_x^u$.
\end{enumerate}
\end{lemma}

A word $x \in A^\omega$ is \emph{linearly recurrent (with constant $K$)} if for all $u \in \Fact x$ and all $w \in \Ret_x(u)$ we have $|w| \leq K|u|$.
It is \emph{periodic} if $x = x_{[n, \infty)}$ for some $n \geq 1$, and \emph{aperiodic} otherwise.
We need the following facts about linearly recurrent words, which can be found in \cite{DuHoSk99}.

\begin{lemma}[Theorem 24 in \cite{DuHoSk99}]
\label{lem:LinRec}
Let $x \in A^\omega$ be an aperiodic linearly recurrent word with constant $K$.
\begin{enumerate}
\item For all $u \in A^+$ we have $u^{K+1} \notin \Fact x$.
\item For all $u \in \Fact x$ and $w \in \Ret_x(u)$ we have $K |w| > |u| > |w|/K$
\item For all $u \in \Fact x$ we have $|\Ret_x(u)| \leq K(K+1)^2$.
\end{enumerate}
\end{lemma}

Important classes of linearly recurrent words include certain \emph{Sturmian words}, and all aperiodic \emph{fixed points of primitive substitutions} \cite{DuHoSk99}.
Namely, a Sturmian word is linearly recurrent if and only if the continued fraction expansion of its asymptotic density of $1$s has bounded coefficients (see \cite{Ca01} for details).
Recall that a substitution $\sigma : A^* \to A^*$ is primitive if there exists $n \in \N$ such that all $a \in A$ occur in $\sigma^n(b)$ for any $b \in A$.

\section{Main Results}

In this section, we present our main result.

\begin{theorem}
If $x \in A^\omega$ is linearly recurrent with constant $K$ and aperiodic, then $x \in \PPm$.
More specifically, we have $x \in \PPm(K+1,k)$ for $k = 2 + \sum_{i=0}^{K^5-1} 2 K^i(K+1)^{2 i}$.
\end{theorem}

\begin{proof}
Denote $p_n = x_{[0, K^n-1]}$, the prefix of $x$ of length $K^n$, and let $B = \{0, \ldots, K(K+1)^2-1\}$.
We use $B$ as a common index set for the return words of every prefix $p_n$, since it contains $R_x(p_n)$ for all $n \in \N$.
Consider a word $u \in \Fact x$.
If $u \in \Ret_x(p_n)^m$ for some $m \in \N$, then Lemma~\ref{lem:Return} implies that there exists a unique word $r \in B^+$ such that $u = \Theta_x^{p_{n-1}}(r)$.
We claim that $|r| < m K^3$.
This is because $|w| > |p_{n-1}|/K = K^{n-2}$ holds for every $w \in \Ret_x(p_{n-1})$ by Lemma~\ref{lem:LinRec}, so that we have $\Theta_x^{p_{n-1}}(r) > m K^{n-2} |r|$, while $|u| \leq m K |p_n| = m K^{n+1}$.
We denote this $r$ by $r(u)$.
Finally, denote by $n(u)$ the unique number $n \in \N$ such that $K^n \leq |u| < K^{n+1}$.

Using the notions defined in the above paragraph, we define the following factor coloring $c : \Fact x \to C$ of $x$, where $C = \{\#, \$\} \cup (\Z_2 \times B^{< K^5})$.
\[
c(u) = \left\{ \begin{array}{ll}
	(n(u) \bmod 2, r(u)), & \mathrm{if~} u \in \Ret_x(p_{n(u)})^{\leq K^2}, \\
	\$, & \mathrm{if~} u \notin \Ret_x(p_{n(u)})^{\leq K^2} \mathrm{~and~} u \in \Pref x, \\
	\#, & \mathrm{otherwise.}
\end{array} \right.
\]
Note that $|C| = 2 + \sum_{i=0}^{K^5-1} 2 K^i(K+1)^{2 i}$.

We claim that no prefix of $x$ admits a monotone and $c$-monochromatic factorization of length $K + 1$.
Suppose on the contrary that such a factorization exists: $U_0 U_1 \cdots U_K \in \Pref x$ where $0 < |U_i| \leq |U_j|$ for all $i < j$, and every $U_i$ has the same $c$-image $\sigma \in C$.
Since $U_0 \in \Pref x$, we have $\sigma \neq \#$.
In particular, every $U_i$ is a prefix of $x$.
Let $n = n(U_0)$.
Since $U_1 \in \Pref x$ and $|U_1| \geq |U_0|$, the word $p_n$ is a prefix of $U_1$, and thus $U_0 \in \Ret_x(p_n)^m$ for some $m \in \N$.
We have $|U_0| < K^{n+1}$ by the definition of $n$, and Lemma~\ref{lem:LinRec} implies $|w| \geq K^{n-1}$ for all $w \in \Ret_x(p_n)$, so that $m \leq K^2$.
By the definition of $c$, we then have $\sigma = c(U_0) = (n \bmod 2, r(U_0))$.

The idea of the rest of the proof is the following.
If every $U_i$ has the same $n$-image, then they are equal by the injectivity of $\Theta_x(p_{n-1})$, so that $x$ has a $(K+1)$-fold repetition as a factor.
If this is not the case, then the $n$-image of one of the words is at least $2$ larger than that of $U_0$, so that it is at least $K$ times longer than $U_0$.
Since that word is also a prefix of $x$, it contains a $(K+1)$-fold repetition of $U_0$.
In both cases, we obtain a contradiction with Lemma~\ref{lem:LinRec}.

We proceed with the proof.
Suppose first that $n(U_0) = n(U_i)$ for all $i \in \{1, 2, \ldots, K\}$.
In this case, we have
\begin{equation}
\label{eq:Same}
U_i = \Theta_x^{p_{n-1}}(r(U_i)) = \Theta_x^{p_{n-1}}(r(U_0)) = U_0
\end{equation}
for all $i$, since $r(U_i) = r(U_0)$, so that the words of the factorization are in fact equal.
But then we have $U_0 U_1 \cdots U_K = U_0^{K+1} \in \Fact x$, contradicting Lemma~\ref{lem:LinRec}.

Suppose then that there exists $1 \leq i \leq K$ such that $n(U_i) > n(U_0)$, and choose the minimal such $i$.
By equation~\eqref{eq:Same}, we again have $U_0 U_1 \cdots U_{i-1} = U_0^i$.
Since $n(U_i) \equiv n(U_0) \bmod 2$, we actually have $n(U_i) \geq n(U_0) + 2$, which implies $|U_i| \geq K|U_0|$.
But since $U_i$ is a prefix of $x$, the word $U_0^i$ is a prefix of $U_i$.
Then the word $U_0 U_1 \cdots U_{i-1} U_0^i = U_0^{2 i}$ is a prefix of $x$, and the longer one of $U_0^{2 i}$ and $U_i$ is a prefix of the other.
Iterating this argument, we see that $U_i$ is a prefix of the infinite periodic word $U_0^\omega$, and since $|U_i| \geq K|U_0|$, the word $U_0^K$ is a prefix of $U_i$.
Since we also have $U_0^i U_i \in \Pref x$ and $i \geq 1$, this implies $U_0^{K+1} \in \Fact x$, again a contradiction.
\end{proof}

Since every fixed point of a primitive substitution is linearly recurrent, we have the following corollary.

\begin{corollary}
If $x \in A^\omega$ is the fixed point of an aperiodic primitive substitution, then $x \in \PPm$.
\end{corollary}

\bibliographystyle{plain}
\bibliography{../../../bib/bib}{}

\end{document}